\documentclass{amsart}
\usepackage[T1]{fontenc}   

\newtheorem{theorem}{Theorem}[section]

\theoremstyle{definition}

\numberwithin{equation}{section}

\begin{document}

\vspace{0.5in}

\renewcommand{\bf}{\bfseries}
\renewcommand{\sc}{\scshape}
\vspace{0.5in}

\title[Cardinality Estimations of Finite Sets]%
{Cardinality Estimations of Sets with Interval Uncertainties in Finite Topological Spaces}


\author{J.F. Peters}
\address{Department of Electrical and Computer Engineering, University of Manitoba and Department
of Mathematics, Faculty of Arts and Sciences, Adýyaman University, 02040
Adýyaman, Turkey}
\thanks{The research has been supported by the Natural Sciences \&
Engineering Research Council of Canada (NSERC) discovery grant 185986, Instituto Nazionale di Alta Matematica (INdAM) Francesco Severi, Gruppo Nazionale  per le Strutture Algebriche, Geometriche e Loro Applicazioni grant 9 920160 000362, n.prot U 2016/000036 and Scientific and Technological Research Council of Turkey (T\"{U}B\.{I}TAK) Scientific Human
Resources Development (BIDEB) under grant no: 2221-1059B211301223.}
\email{james.peters3@ad.umanitoba.ca}

\author{I.J. Dochviri}
\address{Department of Mathematics; Caucasus International University;
Tbilisi, Georgia}
\email{iraklidoch@yahoo.com}





\begin{abstract} In this paper, we have established boundaries of cardinal numbers of nonempty sets in finite non-$T_1$
topological spaces using interval analysis.  For a finite set with known cardinality, we give interval estimations based on the closure and interior of the set.   
In this paper, we give new results for the cardinalities of non-empty semi-open sets in non-$T_1$ topological spaces as well as in extremely disconnected and hyperconnected topological spaces.
\end{abstract}

\maketitle

\section{\bf Introduction}

In the point-set topology trends of last five decades are connected with investigations of topological spaces and their relations with infinite
cardinal functions (see e.g. {\cite {Juh}}). Most of modern topological papers are concerned with infinity type cardinal functions, but many
interesting properties of finite topological spaces are in the shadow. In the same time discrete mathematics and combinatorics uses finite sets
for naturally appeared analytical questions (see e.g. {\cite {EKR}}, {\cite {Ka}}, {\cite {An}}, {\cite {Do}}, {\cite {Do1}}). Here we should
mention well-known theory developed by the P. Erd\"{o}s, C. Ko and R. Rado, where the main thing is cardinality counting problem in the given
finite set. Also, importantly, characterizations of finite sets are widely used in the computer science {\cite {NP}}, {\cite {D}} and
probability theory {\cite {NW}}.

Historically, cardinality counting problems in discrete mathematics and combinatorics began after the introduction of the well-known inclusion-exclusion formula.
For two given finite nonempty sets $A,B$, we have $card (A\cup B)= card(A)+card(B)-card(A\cap B)$.   This inclusion-exclusion formula is applicable, provided the exact values of cardinalities are known (see e.g. {\cite {Do}}, {\cite {DP}}). However, estimations of the cardinalities of corresponding sets may be necessary in cases where only imprecise set-cardinality information about cardinals of involved sets is available.   Such a situation arises when we consider big data sets, molecular structures and so on.   In this paper, we give results for cardinal estimations for the closure and interior of nonempty sets and nonempty semi-open sets in non-$T_1$ topological spaces as well as for nonempty sets in extremely disconnected and hyperconnected topological spaces.

\section{\bf Preliminaries}

In the {\cite {Mo}}, R. Moore developed interval mathematics for computational problems, where parameters of investigating models are uncertain and
we are only able to describe parameters by closed interval estimations.
Below we shortly recall basic operations of interval arithmetic.

Let $a_1, a_2, b_1, b_2, x\in \mathbb{R}$. A closed interval of the reals is denoted by $[a_1, a_2]=\{x\in \mathbb{R}: a_1\leq x\leq a_2\}$.
From {\cite {Mo}}, we have following interval arithmetic:

\bigskip

(1) $[a_1, a_2]+[b_1, b_2]=[a_1+b_1, a_2+b_2]$;

\smallskip

(2) $[a_1, a_2]-[b_1, b_2]=[a_1-b_2, a_2-b_1]$;

\smallskip

(3) $[a_1, a_2]\times [b_1, b_2]=[minP, maxP]$, where $P=\{a_1b_1, a_1b_2, a_2b_1, a_2b_2\}$;

\smallskip

(4) If $0\notin [b_1, b_2]$, then $\frac {[a_1, a_2]}{[b_1, b_2]}=[\frac{a_1}{b_2}, \frac{a_2}{b_1}]$.

\bigskip

It should be especially notice that any real number $k$ is identified with interval $[k, k]$. Moreover, if $a_1$ and $b_1$ are non-negative real
numbers then interval multiplication (3) should be change in the following way $[a_1, a_2]\times [b_1, b_2]=[a_1b_1, a_2b_2]$.

\smallskip

There are established several important computational differences of the interval arithmetic from real one, but we do not need more information than
we present here about interval mathematics.

\bigskip

Below the sets of natural and rational numbers are denoted by symbols $\mathbb{N}$ and $\mathbb{Q}$, but $\mathbb{N}_0=\mathbb{N}\cup \{0\}$.

\smallskip

For a rational number $q\in \mathbb{Q}$ we have to use two well-known notations:

\smallskip

$\lfloor q \rfloor =max \{m\in\mathbb{Z} | m\leq q\}$ and $\lceil q \rceil =min \{n\in\mathbb{Z} | n\geq q\}$.

\bigskip

For topological spaces we use notions from {\cite {AP}}. If $O\subset X$ is nonempty open subset of a topological space $(X,\tau)$ then
we will write $O\in\tau\backslash\{\emptyset\}$. Also, in a topological space $(X,\tau)$ denote by $cl(A)$ closure (resp. $int(A)$ interior of)
$A\subset X$, which is minimal closed (resp. maximal open) set containing (resp. contained in) a set $A$. Recall that a topological space $(X,\tau)$ is $T_1$ space if and only if $\{x\}$ is closed set, for every $x\in X$.
Therefore, in $T_1$ space $(X,\tau)$ we have $\{x\}=cl(\{x\})$. For the finite non-$T_1$ topological spaces cardinal estimations using closure and interior
operators is less lightened part of extremal set theory {\cite {Ju}}.

\smallskip

Naturally, if we know about a set $A$ that both of estimations
$card(A)\in [a_1, a_2]$ and $card(A)\in [b_1, b_2]$ are valid, where $[a_1, a_2]\cap[b_1, b_2]\neq\{\emptyset\}$
then we should declare $card(A)\in [\max\{a_1, b_1\},min\{a_2, b_2\}]$.

\begin{theorem} Let $A\subset B$ be subsets of a set $X$ where $card(X)=n$ and $card(A)\in [a_1, a_2]$.
Then $card(B)\in [a_1 +1, n-1]$.

\end{theorem}

\begin{proof} It is obvious that $A\subset B$ implies that $card(A)<card(B)$. Since the minimal value of cardinality of a set $A$ can be equal to $a_1$,
then $a_1+1\leq card(B)$. On the other hand we have, $B\subset X$ and $card(B)<card(X)=n$. Hence it can be write $card(B)\leq n-1$.

\end{proof}

\begin{theorem} Let $A$, $B$ and $C$ be finite subsets of a set $X$ such that $C=A\times B$, $card(C)\in[c_1,c_2]$ and $card(A)\in [a_1, a_2]$.
Then $card(B)\in [\lceil\frac {c_1}{a_2}\rceil, \lfloor\frac {c_2}{a_1}\rfloor]$.

\end{theorem}

\begin{proof}
Since for Cartesian product $C=A\times B$ we can write following cardinal equality: $card(B)=\frac {card(C)}{card(A)}$, then
applying above mentioned operation of the interval division we get $card(B)\in [\frac {c_1}{a_2}, \frac {c_2}{a_1}]\cap \mathbb{N}_0=[\lceil\frac {c_1}{a_2}\rceil, \lfloor\frac {c_2}{a_1}\rfloor]$.

\end{proof}

\begin{theorem}Let $X=A\cup B$ be a finite set with $card(X)\in [m, n]$, but $card(A)\in [a_1, a_2]$ and $card(B)\in [b_1, b_2]$. Then
$card(A\cap B)\in [a_1+b_1-n, a_2+b_2-m]\cap \mathbb{N}_0$.

\end{theorem}

\begin{proof}
Applying famous inclusion-exclusion formula, we can write $card(A\cap B)=card(A)+card(B)-card(A\cup B)=card(A)+card(B)-[m,n]$.
By substitution of given cardinal estimations we obtain $card(A\cap B)\in [a_1+b_1-n, a_2+b_2-m]\cap\mathbb{N}_0$.
\end{proof}

\section{\bf Main Results}

Below we will work with topological spaces which are not even $T_1$ topologies. Examples of such topological spaces are known in the
point-set topology as $T_0$ and $R_0$ spaces.

\begin{theorem} Let a topological space $(X,\tau)$ be a non-$T_1$ space such that $card(X)=n$. Then the closure $card(cl(A))\in [2m, n]$, for
every nonempty $A\subset X$ with $card(A)=m\in [1, \lfloor\frac{n}{2}\rfloor]$.
\end{theorem}

\begin{proof}
It is known that in $T_1$ topological space $(X,\tau)$ with $card(X)=n$ we have $card(cl(x))=1$, for every $x\in X$. Therefore, in view our conditions we conclude that $card(cl(x))>1$, for every $x\in X$.  Hence $n\geq card(cl(x))\geq 2$,
for every $x\in X$. Note that for the set $A=\{a_1,a_2,...,a_m \}$ we can write its closure as following: $cl(A)=cl(a_1)\cup cl(a_2)\cup...\cup cl(a_m)$.
Therefore, the inequalities hold $n\geq card(cl(A))=card(cl(a_1))+card(cl(a_2))+...+card(cl(a_m))\geq 2m$.
\end{proof}

\begin{theorem} Let a topological space $(X,\tau)$ be a non-$T_1$ space with $card(X)=n$ and $A\subset X$ be such that
$card(A)=p\in [\lceil \frac{n}{2}\rceil, n]$.
If $card(cl(x))\in [2, k_x]$, for any point $x\in (X\setminus A)$ then\\

$card(int(A))\in [0, 2p-n]$, if $n\leq\sum_{x\in (X\setminus A)} {k_x}$ \\

and\\

$card(int(A))\in [n-\sum_{x\in (X\setminus A)} {k_x}, 2p-n]$, if $\sum_{x\in (X\setminus A)} {k_x} <n$.

\end{theorem}

\begin{proof}
Assume that $X=\{x_1,x_2,...,x_n\}$ and $A=\{x_1,x_2,...,x_p\}$. It is known that $int(A)=X\setminus cl(X\setminus A)$. Since $card(X\setminus A)=n-p$ then using Theorem 3.1. we can write
$card(cl(X\setminus A))\in [2(n-p), n]$. But, taking into account condition $card(cl(x_i))\in [2, k_i], i=\overline{p+1,n}$ we obtain better estimation than previous, namely:
$card(cl(X\setminus A))=card(cl(x_{p+1}))+card(cl(x_{p+2}))+...+card(cl(x_{n}))\in [2, k_{p+1}]+ [2, k_{p+2}]+...+[2, k_n]=[2(n-p), min\{n, \sum_{i=p+1}^{n} {k_i}\}]$.

It is clear that if $n\leq \sum_{i=p+1}^{n} {k_i}$ then $card(cl(X\setminus A))\in [2(n-p), n]$. Hence we get $card(int(A))\in [0, 2p-n]$.
If $\sum_{i=p+1}^{n} {k_i}<n$ then $card(cl(X\setminus A))\in [2(n-p), \sum_{i=p+1}^{n} {k_i}]$ and we obtain $card(int(A))\in [n-\sum_{i=p+1}^{n} {k_i}, 2p-n]$.
\end{proof}

Recall that a set $A$ of a topological space $(X,\tau)$ is called semi-open if there exists $O\in\tau\backslash\{\emptyset\}$
such that $O\subset A\subset cl(O)$ {\cite {Le}}. The complement of an semi-open set is called semi-closed. The class of all semi-open (resp. semi-closed)
subsets of a space $(X,\tau)$ we denote usually as $SO(X)$ (resp. $SC(X)$). It can be easily to verify that $A\in SO(X)$ if and only if $A\subset cl(intA)$, but
$B\in SC(X)$ if and only if $int(clB)\subset B$.

\begin{theorem} Let $(X,\tau)$ be a non-$T_1$ finite topological space with $card(X)=n$ and $A\in SO(X)$. Then there exists $k\in\mathbb{N}$ such that
$card(A)\in [k+1, 2k-1]$, where $k\in [1, \lfloor \frac{n}{2}\rfloor]$.
\end{theorem}

\begin{proof} For a set $A\in SO(X)$ we can choose $O\in\tau\backslash\{\emptyset\}$ such that $O\subset A\subset cl(O)$.
Hence $card(O)<card(A)< card(cl(O))\leq n$. Denote by $k=card(O)$, then it is obvious that $k\in [1, n-1]$. Hence $card(A)\in [k+1, n-1]$, but by
Theorem 3.1. we can write $card(clO)\in [2k, n]$. Note that the inequality $2k<n$ implies $k\in [1, \lfloor \frac{n}{2}\rfloor]$. Collecting our
estimations we get $k+1<card(A)< [2k,n]$, i.e. $card(A)\in [k+1, 2k-1]$.
\end{proof}

A topological space $(X,\tau)$ is called extremally disconnected (shortly, E.D. space) if $cl(O)\in\tau$, for every $O\in\tau$. It can be
easily verify that $(X,\tau)$ is E.D. topological space if and only if $cl(O_1)\cap cl(O_2)=\emptyset$, for every pair of disjoint open sets $O_1$
and $O_2$ {\cite {AP}}.

\begin{theorem}Let $(X,\tau)$ be a non-$T_1$, E.D. topological space with $card(X)\in [m, n]$ and $A,B\in\tau\backslash\{\emptyset\}$ be the disjoint sets
with $card(A)\in [a_1, a_2]$ and $card(B)\in [b_1, b_2]$. Then
$card(cl(A\cup B))\in [2a_1+2b_1, n]\cap \mathbb{N}_0$.

\end{theorem}

\begin{proof}
First we write $card(cl(A\cup B))=card(cl(A))+card(cl(B))-card(cl(A)\cap cl(B))$. Using E.D. of $(X,\tau)$ we get $card(cl(A)\cap cl(B))=0$.
Therefore we have $card(cl(A\cup B))=card(cl(A))+card(cl(B))$. Now using Theorem 3.1. we get
$card(cl(A\cup B))\in [[2a_1+2b_1, 2a_2+2b_2], [m, n]]\cap\mathbb{N}_0 =[2a_1+2b_1, n]\cap\mathbb{N}_0$.
\end{proof}

A topological space $(X,\tau)$ is called hyperconnected if $cl(O)=X$, for every $O\in\tau\setminus\{\emptyset\}$. It is obvious that $(X,\tau)$ is
hyperconnected if and only if $O_1\cap O_2\neq\{\emptyset\}$, for any pair of $O_1, O_2 \in\tau\setminus\{\emptyset\}$.

\begin{theorem}Let $(X,\tau)$ be a hyperconnected topological space with $card(X)\in [m, n]$. If $X=O_1\cup O_2$, where $O_1, O_2\in\tau\backslash\{\emptyset\}$
are sets with $card(O_1)\in [a_1, a_2]$ and $card(O_2)\in [b_1, b_2]$. Then $card(O_1\cap O_2)\in [a_1+b_1-n, a_2+b_2-m]\cap \mathbb{N}_0$.

\end{theorem}

\begin{proof} Since in the hyperconnected space $(X,\tau)$ we have $O_1\cap O_2\neq\emptyset$, for any pair of $O_1, O_2 \in\tau\setminus\{\emptyset\}$
then it takes place following equality: $card(O_1\cap O_2)=card(O_1)+card(O_2)-card(O_1\cup O_2)=[a_1,a_2]+[b_1,b_2]-[m, n]=[a_1+b_1-n,a_2+b_2-m]\cap\mathbb{N}_0$.

\end{proof}

\bibliographystyle{plain}

\end{document}